\documentclass[12pt]{amsart}

\usepackage[utf8]{inputenc}
\usepackage[T1]{fontenc}
\usepackage{lmodern}
\usepackage[english]{babel}
\usepackage{hyperref}
\usepackage{amsmath,mathtools,amsthm,amssymb,amsfonts}
\usepackage{color}
\usepackage{comment}

\usepackage{mathrsfs}
\usepackage{graphicx}
\usepackage{enumitem}
\usepackage{tikz} 
\usepackage{pgfplots}
\usepackage{esvect}
\usepackage{pdfpages}
\usepackage{nicefrac}
\usepackage{leftidx, tensor}
\usepackage{cleveref}
\usetikzlibrary{shapes,arrows}
\usetikzlibrary{calc}
\usetikzlibrary{shapes.geometric}
\usetikzlibrary{shapes.arrows}
\usetikzlibrary{arrows.meta}
\usetikzlibrary{decorations.markings}
\usetikzlibrary{fit}
\usetikzlibrary{patterns}
\usetikzlibrary{hobby}
\usepgfplotslibrary{patchplots}
\pgfplotsset{compat=1.11}

\renewcommand{\thefootnote}{\fnsymbol{footnote}}
\newcommand{\definedas}{\mathrel{\raise.095ex\hbox{\rm :}\mkern-5.2mu=}}

\newcommand{\R}{\mathbb{R}}

\newcommand{\Sbb}{\mathbb{S}}
\newcommand{\SCal}{\mathcal{S}}

\renewcommand{\d}{\,\mathrm{d}}
\newcommand{\graph}{\,\mathrm{graph}}

\newcommand{\Hess}{\mathrm{Hess}}

\newcommand{\btr}[1]{\left\vert#1\right\vert}
\newcommand{\norm}[1]{\btr{\btr{#1}}}

\newcommand{\spann}[1]{\left\langle#1\right\rangle}

\newcommand{\trace}{\mathrm{tr}}
\newcommand{\Ric}{\mathrm{Ric}}
\newcommand{\scal}{\mathrm{R}}
\newcommand{\Rm}{\mathrm{Rm}}

\newcommand{\two}{\text{II}}
\newcommand{\tr}{\trace}
\newcommand{\dive}{\text{div}}

\theoremstyle{plain}
\newtheorem{thm}{Theorem}[section]
\newtheorem{prop}[thm]{Proposition}

\newtheorem{lem}[thm]{Lemma}

\theoremstyle{definition}

\newtheorem{bem}[thm]{Remark}
\newtheorem{kor}[thm]{Corollary}

\begin{document}
		\begin{center}
			{\LARGE {On effects of the null energy condition on totally umbilic hypersurfaces in a class of static spacetimes}\par}
		\end{center}
		\vspace{0.5cm}
		\begin{center}
			{\large Markus Wolff\footnote[2]{wolff@math.uni-tuebingen.de}}\\
			\vspace{0.4cm}
			{\large Department of Mathematics}\\
			{\large Eberhard Karls Universit\"at T\"ubingen}
		\end{center}
		\vspace{0.4cm}
		\section*{Abstract}
			We study the effects of the null energy condition on totally umbilic hypersurfaces in a class of static spacetimes, both in the spacelike and the timelike case, respectively. In the spacelike case, we study totally umbilic warped product graphs and give a full characterization of embedded surfaces with constant spacetime mean curvature using an Alexandrov Theorem by Brendle and Borghini--Fogagnolo--Pinamonti. In the timelike case, we achieve a characterization of photon surfaces with constant umbilicity factor similar to a result by Cederbaum--Galloway.\\\\
		\renewcommand{\thefootnote}{\arabic{footnote}}
		\setcounter{footnote}{0}
	\section{Introduction}
		The null energy condition, also known as null convergence condition, states that for any null vector field $L$ the Einstein-Tensor $\mathfrak{G}$ satisfies
		\begin{align}\label{introNEC}
			\mathfrak{G}(L,L)\ge 0.
		\end{align}
		In cosmology, this assumption on the spacetime takes a critical role in the formation of singularities (cf. \cite{hawkelli} Section 8.2). Moreover, in a class of static, warped product spacetimes and on their respective time-symmetric slices, it was used to establish versions of a (Spacetime) Alexandrov Theorem \cite{brendle,wangwangzang, wang}.
		In this paper, we study the effects of the (NEC) on both spacelike and timelike totally umbilic hypersurfaces in this class of static spacetimes.
		
		In the spacelike case, we study initial data sets given as warped product graphs and show that in the totally umbilic case an Alexandrov Theorem by Brendle becomes applicable if the null energy condition is satisfied.
		The classical Alexandrov Theorem \cite{aleks} in Euclidean space  states that any compact surface of constant mean curvature is a round sphere. In \cite{brendle}, Brendle generalized this to a large class of Riemannian warped product manifolds. Brendle proved that if certain conditions (H1)-(H3) are satisfied on the manifold, then any orientable, embedded, closed surface of constant mean curvature is necessarily totally umbilic \cite[Theorem 1.1]{brendle}. If an additional assumption (H4) is met Brendle showed that the surface is a leaf of the canonical foliation. After a suitable coordinate change using condition (H2), one can embed the Riemannian manifolds under consideration as time-symmetric $\{t=\operatorname{const.}\}$-slices into a certain Class $\mathfrak{H}$ of static spacetimes of the form
		\[
			\mathfrak{M}=\R\times I\times \mathcal{N},\text{ }\mathfrak{g}=-h\d t^2+\frac{1}{h}\d r^2+r^2g_{\mathcal{N}},
		\]
		where $h$ is a smooth, positive function on some open intervall $I$, and $(\mathcal{N},g_{\mathcal{N}})$ is a compact Riemannian manifold.
		Then, the conditions (H1) and (H3) translate into physically reasonable assumptions on the spacetime, more precisely, that the inner boundary is a non-degenerate Killing horizon and that the spacetime satisfies the null energy condition, see \cite{wangwangzang}. In \cite{wangwangzang}, Wang--Wang--Zhang used the results of Brendle in this spacetime setting to proof a spacetime Alexandrov theorem which characterizes round null cones if they admit a spacelike cross section satisfying a spacetime CMC condition (see \cite{wangwangzang} Section 3.2). Moreover, Wang \cite{wang} noted that  in spherical symmetry condition (H4) is already met under (H1)-(H3), and in a recent paper \cite{borghinifogagnolopinamonti}, Borghini--Fogagnolo--Pinamonti completely removed assumption (H4) as an application for a rigidity statement of the Heintze--Karcher inequality in substatic manifolds.
		
		Regarding the results in \cite{brendle, borghinifogagnolopinamonti}, the question whether these conditions on the spacetime allow for a similar characterization of surfaces not only on the time-symmetric slices, but general initial data sets with warped product structure, naturally arises, and one would expect that the respective {geometric property} of the surfaces also arises out of the structure of the spacetime. 
		Due to its Lorentz invariance, it seems natural to consider a characterization of surfaces that posses constant spacetime mean curvature  $\mathcal{H}^2$ (STCMC surfaces), where $\mathcal{H}^2$ is the Lorentzian length of the mean curvature vector $\vec{\mathcal{H}}$ in the spacetime. If the surface is contained within a time-symmetric initial data set this reduces to the Alexandrov Theorem proven in \cite{brendle, borghinifogagnolopinamonti}. A foliation of such STCMC surfaces at spacelike infinity has recently been used by Cederbaum--Sakovich \cite{cederbaumsakovich} in a center of mass formulation in the asymptotically flat setting. To apply the results of Brendle \cite{brendle}, we utilize a recent construction by Cederbaum and the author \cite{cedwolff}  of generalized Kruskal--Szekeres coordinates.
		
		In the timelike case, we study photon surfaces, i.e., timelike, totally umbilic hypersurfaces, within the same Class $\mathfrak{H}$ of static spacetimes. Imposing an assumption on the eigenvalues of the Ricci tensor as Brendle \cite{brendle} in the Riemannian setting, we obtain a short proof for a characterization of photon surfaces with constant umbilicity factor $\lambda$ in the spacetime setting. In a recent paper, Cederbaum--Galloway \cite{cedgal} achieved a similar characterization of photon surfaces in the same class of spacetimes assuming spherical symmetry and a nowhere locally spacetime conformally flat condition. Although they do not need to assume that $\lambda=\operatorname{const.}$, their arguments rely heavily on the spherical symmetry and can not be extended across a non-degenerate Killing horizon without additional assumptions. We include a detailed discussion and give an illustrative example of a $1$-parameter family within this Class $\mathfrak{H}$ of spacetimes that are locally spacetime conformally flat but satisfy the assumption on the eigenvalues of the Ricci tensor.
		\subsection*{Main results}
		
		\addtocounter{section}{3}
		\setcounter{thm}{7}
		Given a non-negative solution $x$ of an ordinary differential inequality \eqref{mainODI} related to the (NEC), we show that the results in \cite{brendle, borghinifogagnolopinamonti} can also be applied to spacelike warped product graphs with metric coefficient $h+x$ (Theorem \ref{thm_main1}). For some non-negative constant $C$, $x=Cs^2$  is an exact solutions of \eqref{mainODI}. We will refer to the resulting graphs as hyperboloids, and note that they are totally umbilic with constant umbilicity factor (cf. Corollary \ref{kor_hyperboloids}). In particular, the characterization of embedded CMC surfaces allows for a characterization of STCMC surfaces on hyperboloids. Note that this requires to extend the graphs past the non-degenerate Killing horizon into a suitable spacetime extension. Here, we utilize a recent construction of generalized Kruskal--Szekeres coordinates \cite{cedwolff}. See also \cite{brillhayward, schindagui}.
		\begin{thm}
			Let $h:(0,\infty)\to\R$ be a smooth function with finitely many, positive simple zeroes $r_1<\dotsc< r_N$, $(\mathcal{N},g_{\mathcal{N}})$ an $(n-1)$-dimensional Riemannian manifold $(n\ge 3)$. Let $(\mathfrak{M},\mathfrak{g})$ be the corresponding spacetime of Class $\mathfrak{H}$ with metric coefficient $h$ and fibre $\mathcal{N}$. Assume that the generalized Kruskal--Szekeres extension satisfies the (NEC) condition.\newline
			Then there exists a constant $C_0=C_0(h,h')\in(0,\infty]$ such that for any hyperboloid \linebreak $(M_T,g^T,K^T)$ with umbilicity factor $\lambda_T$ satisfying $\lambda_T^2<C_0$ we have:
			If $\Sigma\subset M_T$ is an orientable, closed, embedded hyper\-surface with constant spacetime mean curvature, then $\Sigma$ is a slice $\{s\}\times \mathcal{N}$.
		\end{thm}
		Note that the Kruskal--Szekeres extension of the $(n+1)$-dimensional Schwarzschild spacetime  with positive masss satisfies all of the above assumptions with $C_0=\infty$.
		\begin{kor}
			Let $(\mathfrak{M},\mathfrak{g})$ denote the Schwarzschild spacetime with positive mass. Then any STCMC surface $\Sigma$ in an hyperboloid $(M_T,g^T,K^T)$ is a slice $\{s\}\times \mathbb{S}^{n-1}$.
		\end{kor}
		\addtocounter{section}{1}
		\setcounter{thm}{0}
		Further, if the (NEC) is strict in some sense (cf. Section \ref{sec_photonsurfaces}), this also allows for a characterization of photon surfaces $(\mathcal{P}^n,\mathfrak{p})$, i.e., timelike and totally umbilic hypersurfaces, with constant umbilicity factor.
		\begin{thm}
			Let $(\mathcal{P}^n,\mathfrak{p})$ be a connected photon surface with constant umbilicity factor $\lambda$ in a spacetime of Class $\mathfrak{H}$ or its respective generalized Kruskal--Szekeres extension as constructed in \cite{cedwolff}. Assume further, that $h$ satisfies
			\begin{align*}
			\frac{1}{2}h''+\frac{(n-3)}{2r}h'-\frac{n-2}{r^2}h+r^{-2}\Ric_{g_\mathcal{N}}(X,X)\not= 0
			\end{align*}
			for any unit tangent vector $X$ in $(\mathcal{N},g_\mathcal{N})$ on a dense set of radii in $(0,\infty)$.\newline Then $\mathcal{P}^n$ is either rotationally symmetric, or $\mathcal{P}^n$ is totally geodesic with parallel unit normal vector $\eta$ everywhere tangent to $\mathcal{N}$.
		\end{thm}\newpage
		This paper is structured as follows:
		In Section \ref{sec_prelim} we introduce basic notation and some preliminary Lemmas. In Section \ref{sec_graoh} we study the geometry of spacelike warped product graphs within spacetimes of Class $\mathfrak{H}$. In Section \ref{sec_spacelikesetting} we prove the characterization of STCMC surfaces on spacelike totally umbilic warped product graphs. In Section \ref{sec_photonsurfaces} we prove the characterization of photon surfaces in the timelike case and compare it to a result by Cederbaum--Galloway~\cite{cedgal}.
		\subsection*{Acknowledgements.}
		I would like to express my sincere gratitude towards my supervisor Carla Cederbaum for her guidance and helpfull discussions.
		\setcounter{section}{1}
	\section{Preliminaries}\label{sec_prelim}
	Let $(\mathfrak{M},\mathfrak{g})$ be a spacetime. We denote its Ricci curvature tensor, and scalar curvature by $\mathfrak{Ric}$ and $\mathfrak{R}$, respectively. The Einstein tensor is then given as $\mathfrak{G}\definedas \mathfrak{Ric}-\frac{1}{2}\mathfrak{R}\mathfrak{g}$, and we say that $(\mathfrak{M},\mathfrak{g})$ satisfies the \emph{null energy condition (NEC)} (or \emph{null convergence condition}), if 
	\[
		\mathfrak{G}(L,L)=\mathfrak{Ric}(L,L)\ge 0,
	\]
	for any null vector field $L\in\Gamma(T\mathfrak{M})$. If we consider an initial data set $(M,g,K)$ in $(\mathfrak{M},\mathfrak{g})$, i.e., a spacelike hypersurface $M$ of $\mathfrak{M}$ with induced metric $g$ and second fundamental form $K$ with respect to a future timelike unit normal $\vec{n}$, the Gauß-Codazzi equations imply the well-known \emph{constraint equations} on $(M,g,K)$:
	\begin{align*}
		2\mu&=\operatorname{R}-\btr{K}^2+(\tr_MK)^2,\\
		J&=\dive_M(K-\tr_MKg),
	\end{align*}
	where $\mu\definedas \mathfrak{G}(\vec{n},\vec{n})$, and $J(\cdot)\definedas \mathfrak{G}(\vec{n},\cdot)$ are called the \emph{energy- and momentum density}, respectively, and where $\operatorname{R}$ denotes the scalar curvature of $(M,g)$. Hence, for any unit vector field $V$ on $(M,g)$, the (NEC) implies
	\begin{align*}
		0&\le\mathfrak{Ric}(V\pm\vec{n},V\pm\vec{n})\\
		&=\mathfrak{Ric}(V,XV)\pm 2\mathfrak{Ric}(V,\vec{n})+\mathfrak{Ric}(\vec{n},\vec{n})\\
		&=\mu-\frac{1}{2}\mathfrak{R}\pm 2J(V)+\mathfrak{Ric}(V,V).
	\end{align*}
	
	Note that we use the following conventions for the Riemann curvature tensor $\operatorname{Rm}$, Ricci curvature tensor $\operatorname{Ric}$ and scalar curvature $\operatorname{R}$, respectively:
	\begin{align*}
	\Rm(X,Y,W,Z)&=\spann{\nabla_X\nabla_YZ-\nabla_Y\nabla_XZ-\nabla_{[X,Y]}Z,W},\\
	\Ric(X,Y)&=\tr\Rm(X,\cdot,Y,\cdot),\\
	R&=\tr \Ric.
	\end{align*}\newpage
	
	Troughout this paper, we consider the following class $\mathfrak{H}$ of spacetimes: We say an $(n+1)$-dimensional spacetimes $(\mathfrak{M},\mathfrak{g})$ ($n\ge 3$) is of \emph{class $\mathfrak{H}$ with metric coefficient $h$ and fibre $\mathcal{N}$} if $\mathfrak{M}=\R\times I\times \mathcal{N}$ for an open, non-negative interval $I=(r_H,\infty)$, $r_H\ge 0$, and some compact $(n-1)$-dimensional  Riemannian manifold $(\mathcal{N},g_{\mathcal{N}})$ such that
	\[
	\mathfrak{g}=-h(r)\d t^2+\frac{1}{h(r)}\d r^2+r^2g_{\mathcal{N}},
	\]
	for a smooth function $h\colon (0,\infty)\to \R$ that is strictly positive on $I$ and satisfies $h(r_H)=0$, if $r_H>0$. Hence, the inner boundary $\{r=r_H\}$ is a Killing horizon of $(\mathfrak{M},\mathfrak{g})$. We will denote coordinates on $\mathcal{N}$ with capital Roman letters. In spherical symmetry, i.e., ${(\mathcal{N},g_{\mathcal{N}})=(\Sbb^{n-1},\d\Omega^2)}$, where $\d\Omega^2$ denotes the round metric on $\Sbb^{n-1}$, these spacetimes compose a large class of physically significant models both in the context of isolated gravitating systems as well as in the context of cosmology, such as the Schwarzschild and Reissner-Nordström spacetime, and the de\,Sitter and anti-de\,Sitter spacetime. The properties of this class of spherically symmetric spacetimes are well understood and have been subject to extensive research, see e.g. Cederbaum--Galloway \cite{cedgal},  Schindler--Aguirre \cite{schindagui}. Note that the more general form has also been considered in other works, see e.g. Wang--Wang--Zhang \cite{wangwangzang}, Brill--Hayward \cite{brillhayward}, and a joint work of Cederbaum and the author \cite{cedwolff}. If we assume that $(\mathcal{N},g^{\mathcal{N}})$ has constant sectional curvature, then a spacetime of class $\mathfrak{H}$ is equipped with a Birmingham--Kottler metric, see e.g. \cite{birmi, kottler, chrugalpot}. In spherical symmetry we will adopt the notion of Cederbaum--Galloway and call them spacetimes of class $\SCal$.
	
	Let $(M_0, g^0)$ denote the time-symmetric ($K\equiv0$) time slice $\{t=0\}$ with induced metric
	\[
		g^0 =\frac{1}{h(r)}\d r^2+r^2g_\mathcal{N}.
	\]
	Note that the warped product manifolds considered by Brendle \cite{brendle} are precisely of the form $(M_0,g^0)$ after a change of coordinates if condition (H2) imposed by Brendle is satisfied. Then condition (H1) implies that $\{r=r_H\}$ is a non-degenerate Killing horizon in $(\mathfrak{M},\mathfrak{g})$ in the sense that $h'(r_H)\not=0$, cf. \cite[ Equation (12.5.16)]{wald}. Moreover, Wang--Wang--Zhang \cite{wangwangzang} pointed out that condition (H3) of Brendle precisely translates to $(\mathfrak{M},\mathfrak{g})$ satisfying the null energy condition. For the convenience of the reader, we collect the respective equivalences in the notation of this paper in the following lemma:
	\begin{lem}\label{lemma_NEC}
		Let $(\mathfrak{M},\mathfrak{g})$ be a spacetime of class $\mathfrak{H}$ with metric coefficient $h$, and let $f\definedas \sqrt{h}$ on $(r_H,\infty)$. Then the following are equivalent:
		\begin{enumerate}
			\item[\emph{(i)}] $(\mathfrak{M},\mathfrak{g})$ satisfies the (NEC).
			\item[\emph{(ii)}] $\Delta_0fg^0-\Hess_0f+f\Ric^0\ge 0$ on $M_0$, where $\Delta_0$, $\Hess_0$, and $\Ric^0$ denote the Laplacian, the Hessian and Ricci curvature with respect to $g^0$, respectively.
			\item[\emph{(iii)}] It holds
			\[
			\frac{1}{2}h''+\frac{(n-3)}{2r}h'-\frac{n-2}{r^2}h+r^{-2}\Ric_{g_\mathcal{N}}(X,X)\ge 0
			\]
			on $M_0$ for all unit tangent vector fields $X$ in $(\mathcal{N},g_{\mathcal{N}})$, where $\Ric_{g_\mathcal{N}}$ denotes the Ricci curvature on $\mathcal{N}$.
			\item[\emph{(iv)}] The function $x=h-\alpha\colon (r_H,\infty)\to \R$ is a solution of the ordinary differential inequality
			\[
			\frac{1}{2}x''+\frac{(n-3)}{2s}x'-\frac{(n-2)}{s^2}x\ge 0,
			\]
			where $(n-2)\alpha$ is the minimum of the smallest eigenvalue of $\Ric_{g_\mathcal{N}}$ on $\mathcal{N}$.
		\end{enumerate}
	\end{lem}
	For the equivalences (i) to (iii), we refer to the respective results of Brendle and Wang--Wang--Zhang, see Proposition 2.1. in \cite{brendle} and Lemma 3.8 in \cite{wangwangzang}, where $\alpha$ is the same as the constant $\rho$ considered by Brendle \cite{brendle} as above\footnote{Note the different conventions for the functions $f$, $h$ used here compared to \cite{brendle} and \cite{wangwangzang}.}. The fourth equivalence is immediate, since we assume $\mathcal{N}$ to be compact. We moreover observe that
	\[
	L^*_g(f)=-\Delta_0fg^0+\Hess_0f-f\Ric^0,
	\]
	where $L^*_g$ denotes the formal $L^2$ adjoint of the linearization of the scalar curvature operator, cf.  \cite{corvino} Lemma 2.2 . Thus, the (NEC) implies the existence of a non-trivial supersolution $f>0$ of the formal $L^2$ adjoint of the linearization of the scalar curvature operator. As equality for the (NEC) implies a non-trivial kernel, we can conclude from \cite{corvino} Lemma 2.3  that then the scalar curvature $R^0$ of the time-symmetric slices must necessarily be constant in this case. See also Remark \ref{bem_NEC} below.
	
	Wang \cite{wang} further noticed that the (NEC) is in fact also related to an eigenvalue analysis of the Ricci curvature tensor  $\Ric^0$ of the time-symmetric time slices. Namely, the (NEC) implies monotonocity for the difference between the eigenvalue $h\Ric^0_{rr}$ and any eigenvalue of $\Ric^0\vert_{T\mathcal{N}\times T\mathcal{N}}$, cf. in \cite[Lemma 5.1]{wang}. Analogous to spherical symmetry, we can establish this monotonicity in the general case by direct computation:
	\begin{lem}\label{lem_monotonicity}
		Let $(\mathfrak{M},\mathfrak{g})$ be a spacetime of class $\mathfrak{H}$ satisfying the (NEC), and let $X$ be a unit tangent vector in $(\mathcal{N},g_{\mathcal{N}})$. Then
		\begin{align}\label{eq_lemmonotone}
		r^n\left(\frac{(n-2)}{2r}h'-\frac{(n-2)}{r^2}h+r^{-2}\Ric_{g_\mathcal{N}}(X,X)\right)
		\end{align}
		is monotone non-decreasing in $r$.
	\end{lem}
	Thus, this holds true for any unit eigenvector $X\in\Gamma(T\mathcal{N})$ of $\Ric_{g_\mathcal{N}}$, in particular for the minimum $(n-2)\alpha$, and condition (H4) of Brendle \cite{brendle} is equivalent to the monotone quantity \eqref{eq_lemmonotone} being strictly positive everywhere. Due to the monotonicity, it suffices to check this at the inner boundary, so condition (H4) is in particular implied by the boundary condition 
	\[
	h'(r_H)r_H+2\alpha>0,
	\]
	and hence immediate for $\alpha\ge 0$, cf. \cite[Remark 5.2]{wang}.
	Note that it suffices to assume that \eqref{eq_lemmonotone} is non-vanishing, cf. (H4') \cite[page 3]{brendle}. In a recent paper \cite{borghinifogagnolopinamonti}, Borghini--Fogagnolo--Pinamonti  obtain a similar rigidity statement in Class $\mathfrak{H}$ for strictly mean convex surfaces that satisfy equality for the substatic Heintze--Karcher inequality assuming (H1)--(H3), cf. \cite[Theorem 1.2]{borghinifogagnolopinamonti}. As already remarked by Brendle \cite[Section 6]{brendle}, CMC surfaces satisfy equality for the substatic Heintze--Karcher inequality which yields the desired Alexandrov Theorem even without imposing any version of condition (H4), cf. \cite[Corollary 1.3]{borghinifogagnolopinamonti}. As they need to verify an additional technical assumption in the case that $\Sigma$ is homologous to $\partial M$ to apply a splitting theorem \cite[Theorem 3.1]{borghinifogagnolopinamonti}, they state their results only in this case, whereas the nullhomologous case follows directly from \cite[Theorem 3.1]{borghinifogagnolopinamonti}.
	
	We close this section by briefly introducing our notation for spacelike codimension-$2$ surfaces $(\Sigma,\gamma)$ in a spacetime $(\mathfrak{M},\mathfrak{g})$. Recall that the \emph{vector-valued second fundamental form} $\vec{\two}$ of $\Sigma$ in $(\mathfrak{M},\mathfrak{g})$ is defined as
	\[
		\vec{\two}(V,W)=\left(\overline{\nabla}_VW\right)^\perp
	\]
	for all tangent vector fields $V, W\in\Gamma(T\Sigma)$, where $\overline{\nabla}$ denotes the Levi-Civita connection of $(\mathfrak{M},\mathfrak{g})$. Then, the \emph{codimension-$2$ mean curvature vector} $\vec{\mathcal{H}}$ of $\Sigma$ is given by the trace of $\vec{\two}$ with respect to $\gamma$, i.e, $\vec{\mathcal{H}}=\tr_\gamma\vec{\two}$. In the following, we will always assume that $(\Sigma,\gamma)$ is closed, and embedded as a hypersurface in an initial data set $(M,g,K)$, with second fundamental form $K$ with respect to a future timelike unit normel $\vec{n}$. Let $\nu$ denote a unit normal of $(\Sigma,\gamma)$ in $(M,g)$. Then $\vec{\mathcal{H}}$ admits the decomposition
	\[
		\vec{\mathcal{H}}=-H\nu+P\vec{n},
	\]
	where $H$ denotes the mean curvature of $\Sigma$ in $(M,g)$, and $P\definedas\tr_\Sigma K=\tr_MK-K(\nu,\nu)$. Moreover, the \emph{expansions} $\theta_\pm$ with respect to the null directions $l_\pm=\nu\pm\vec{n}$ are given by
	\[
		\theta_\pm\definedas H\pm P.
	\]
	Here, we define the \emph{spacetime mean curvature} of $\Sigma$ as
	\[
		\mathcal{H}^2\definedas H^2-P^2.
	\]
	Assuming that $\mathcal{H}^2>0$, this agrees with the notion of spacetime mean curvature by Ceder\-baum--Sakovich \cite{cederbaumsakovich} upon taking a square root. However, in general $\mathcal{H}^2$ might be at least locally negative. Indeed, surfaces where $\mathcal{H}^2<0$ everywhere along $\Sigma$ naturally occur in our setting, cf. Remark \ref{bem_STCMCspheres}, and are called \emph{trapped} in the context of General Relativity.
	Note that the expansions $\theta_\pm$ are the change of area of $\Sigma$ in the null directions $l_\pm$, and the spacetime mean curvature is the Lorentzian length of $\vec{\mathcal{H}}$, i.e.,
	\[
		\mathfrak{g}(\vec{\mathcal{H}},\vec{\mathcal{H}})=\mathcal{H}^2=\theta_+\theta_-.
	\]
	We say a surface $(\Sigma,\gamma)$ in an initial data set $(M,g,K)$ has \emph{constant expansion} (CE), and \emph{constant spacetime mean curvature} (STCMC), if $\theta_\pm=\operatorname{const.}$, and $\mathcal{H}^2=\operatorname{const.}$ respectively. In the special case when $\theta_\pm=0$, and $\mathcal{H}^2=0$, we call $\Sigma$ a \emph{marginally outer/inner trapped surfaces (MOTS/MITS)}, and a \emph{generalized apparent horizon}, respectively. Any MOTS/MITS is always a generalized horizon, but the converse is not true in general. See Carrasco--Mars \cite{carmars} for an explicit counterexample.
	
	\section{Warped product graphs in Class $\mathfrak{H}$}\label{sec_graoh}
	We consider spacelike warped product graphs over the canonical $\{t=0\}$ time slice $M_0$ in spacetimes of Class $\mathfrak{H}$. More precisely, we look at hypersurfaces $M_T$ of the form
	\[
	M_T={\{(T(s),s)\colon s\in(r_1,r_2)\}\times \mathcal{N} }
	\] 
	for some smooth function $T:(r_1,r_2)\to\R$ with $r_H\le r_1<r_2\le\infty$. We will refer to $T$ as the \emph{radial height function} of $M_T$. We further denote the induced metric and second fundamental form of $M_T$ as $g^T$ and $K^T$, respectively.
	
	Note that any spacelike slice in a static spacetime can always be written as a graph, so the above assumption is only restrictive in the sense that we assume that $T$ is only depending on $r$. For general graphical initial data sets $(M_T,g^T,K^T)$ given as $M_T=\graph_{M_0}T$, the spacelike condition yields a restriction on the gradient of $T$, i.e., that $1-h\btr{\nabla_0T}^2>0$, where $\nabla_0$ denotes the gradient on $M_0$. Using the computations of Cederbaum--Nerz \cite{cednerz} for graphs in general static spacetimes with coordinates $\{x^i\}$ on $M_0$, we get
	\begin{align*}
	\partial_i^T=\partial_i+T_{,i}\partial_t,\\
	g^T_{ij}=g_{ij}-hT_{,i}T_{,j},
	\end{align*}
	and the future timelike unit normal $\vec{n}$ is given by $\vec{n}=\frac{\partial_t+h\nabla_0T}{f\sqrt{1-h\btr{\nabla_0T}^2}}$ with $f\definedas\sqrt{h}$ as above. Moreover, the second fundamental form $K^T$ is given by\footnote{Note that there is a slight mistake in the formula in \cite{cednerz} which has been corrected in \cite{cederbaumsakovich}}
	\begin{align*}
	K^T(\partial_i^T,\partial_j^T)=\frac{f\Hess_0T(\partial_i,\partial_j)+T_{,i}f_{,j}+f_{,i}T_{,j}-h\spann{\nabla_0T,\nabla_0f}T_{,i}T_{,j}}{\sqrt{1-h\btr{\nabla_0T}^2}}.
	\end{align*}
	If $M_T=\{(T(s),s,x^I)\}$ embeds in Class $\mathfrak{H}$ with coordinates $\{s,x^I\}$, where $x^I$ denote (local) coordinates on $\mathcal{N}$, this yields
	\begin{align}
	\begin{split}\label{eq_graphmetric}
	g^T=\frac{1}{h_T(s)}\d s^2+s^2g_{\mathcal{N}},\\
	K^T=a_T(s)\d s^2+b_T(s)s^2g_{\mathcal{N}},
	\end{split}
	\end{align}
	with $s\equiv r$ along $M_T$, and where 
	\begin{align}\label{eq_graph2ndFF1}
	h_T&=\frac{h}{1-h\btr{\nabla_0T}^2},
	\end{align}
	\begin{align}\label{eq_graph2ndFF2}
	a_T&=\frac{fT''+f'T'(3-h\btr{\nabla_0T}^2)}{\sqrt{1-h\btr{\nabla_0T}^2}},
	\end{align}
	\begin{align}\label{eq_graph2ndFF3}
	b_T&=\frac{f^3T'}{s\sqrt{1-h\btr{\nabla_0T}^2}},
	\end{align}
	with the tangent vector fields $\partial_s=\partial_r+T'\partial_t$, $\partial_I$ and the future unit normal 
	\[\vec{n}^T=\frac{\partial_t+h\nabla_0T}{f\sqrt{1-h\btr{\nabla_0T}^2}}.\]
	In particular, we see from \eqref{eq_graphmetric} that $(M_T,g^T)$ is again a warped product manifold as considered in \cite{brendle, borghinifogagnolopinamonti} with $h_T\ge h$, and note that $K$ also satisfies a similar block diagonal form. Thus, we will refer to spacelike graphs $(M_T,g^T,K^T)$ such that $T=T(s)$ as (spacelike) warped product graphs.
	The main observation in this subsection is to see that both the intrinsic and extrinsic curvature for such spacelike warped product graphs are fully determined by the difference $h_T-h$ in Class $\mathfrak{H}$. This essentially follows from the following lemma:
	\begin{lem}\label{lemma_difference}
		\begin{align*}
		b_T^2&=\frac{h_T-h}{s^2},\\
		h_Ta_Tb_T&=\frac{h_T'-h'}{2s}.
		\end{align*}
	\end{lem}
	\begin{proof}
		By \eqref{eq_graph2ndFF1} and using that $\btr{\nabla_0T}^2=h\cdot(T')^2$, we see that
		\begin{align}\label{prooflemma_difference_1}
		h_T-h=\frac{h}{1-h\btr{\nabla_0T}^2}-h=\frac{h^3\cdot(T')^2}{1-h\btr{\nabla_0T}^2}.
		\end{align}
		Thus, the first identity follows by taking a square of \eqref{eq_graph2ndFF3}. Taking a derivative of \eqref{prooflemma_difference_1} the second identity follows from straightforward computation.
	\end{proof}
	\begin{bem}\label{rem_destriptionbydifference}
		Further, the difference $h_T-h$ also uniquely determines the radial height function $T$ up to a choice of sign of the derivative and a constant of integration, as \[\btr{T'}=\frac{1}{h}\sqrt{\frac{h_T-h}{h_T}}.\]
		More precisely $(M_T,g^T,K^T)$ is fully determined by the choice of function $b_T$ with ${h_T=h+r^2b_T^2}$ and 
		\[
		T'=\frac{rb_T}{h\sqrt{h+r^2b_T^2}},
		\]
		up to a constant of integration.
	\end{bem}
	As a consequence, this rigid structure yields a full characterization of totally umbilic spacelike warped product graphs in Class $\mathfrak{H}$:
	\begin{kor}\label{kor_hyperboloids}
		Let $(M_T,g^T,K^T)$ be a spacelike warped product graph as above, and we further assume that $K^T=\lambda_T g^T$ for some smooth function $\lambda_T$. Then $\lambda_T$ is constant, and $(M_T,g^T,K^T)$ is fully determined by the choice $b_T=\lambda_T$ up to a shift in $t$-direction.
	\end{kor}
	In the Minkowski spacetime, where $h=1$, we precisely recover the connected components of an hyperboloid centered around the origin with $h_T=1+\lambda_T^2s^2$, where the sign of $\lambda_T$ determines the choice of connected component. As a totally umbilic, spacelike warped product graph now corresponds to $h_T=h+\lambda_T^2s^2$ in the general case, we will similarly refer to such a graphs as a \emph{hyperboloid} in Class $\mathfrak{H}$.
	\begin{proof}
		Since $K^T=\lambda_T g^T$, we have
		\[
		h_Ta_T=\lambda_T=b_T,
		\]
		in particular
		\[
		h_Ta_Tb_T=b_T^2.
		\]
		Using Lemma \ref{lemma_difference}, we see that
		\[
		(h_T-h)'=\frac{2}{s}(h_T-h),
		\]
		and solving the ODE gives $h_T-h=Cs^2$, where necessarily $C\ge0$ since $h_T\ge h$. By Lemma~\ref{lemma_difference}
		\[
		\lambda_T^2=b_T^2=C,
		\]
		so $\lambda_T$ is constant. As $T$ is uniquely determined by $b_T=\lambda_T$ up to a constant of integration, the claim is proven.
	\end{proof}
	\begin{bem}
		Similarly, we can characterize all spacelike warped product graphs with \linebreak${\tr_{M_T}K^T\equiv C}$ via 
		\[
		h_T=h+\left(\frac{C}{n}s+\frac{c_1}{s^{n-1}}\right)^2
		\]
		for some real constant $c_1\in\R$. The choice $c_1=0$ corresponds to the totally umbilic case and we recover the hyperboloids. For $C=0$, we obtain a $1$-parameter family of maximal hypersurfaces. These CMC graphs have been considered by Bartnik--Simon \cite{bartniksimon} as barries in the Minkowski spacetime. See also \cite{bartik, gerhardt, LeeLee}.
	\end{bem}\newpage
	\begin{bem}
		Although the assumption of being a graph in the above sense is much more restrictive in the case of a timelike hypersurface, we can establish a similar warped product structure for timelike graphs with radial height function $T=T(s)$ where we now require that ${h\btr{\nabla_0T}^2-1>0}$. More precisely, we find that
		\begin{align*}
		g_T&=-\frac{1}{h_T}\d s^2+s^2g_{\mathcal{N}},\\
		K_T&=a_T\d s^2+b_Ts^2g_{\mathcal{N}},
		\end{align*}
		with 
		\begin{align*}
		h_T&=\frac{h}{h\btr{\nabla_0T}^2-1},\\
		a_T&=\frac{fT''+f'T'(3-h\btr{\nabla_0T}^2)}{\sqrt{h\btr{\nabla_0T-1}^2}},\\
		b_T&=\frac{f^3T'}{s\sqrt{h\btr{\nabla_0T}^2-1}},
		\end{align*}
		and find the relations
		\begin{align*}
		b_T^2&=\frac{h_T+h}{s^2},\\
		h_Ta_Tb_T&=-\frac{(h_T+h)'}{2s}.
		\end{align*}
		In the totally umbilic case, this leads to the same ODE system characterizing rotationally symmetric photon surfaces in class $\mathcal{S}$ derived by Cederbaum--Galloway \cite{cedgal}\footnote{This has been extended to Class $\mathfrak{H}$ by Cederbaum and the author in \cite{cedwolff}.}. In \cite{cedoliviasophia}, Cederbaum--Jahns--Vi\v{c}\'{a}nek Mart\'{i}nez fully characterize the behavior of solutions to this ODE, in particular showing that rotationally symmetric photon surfaces are either photon spheres or warped product graphs of the above sense away from singular radii.
		
		In particular $h_T=\lambda_T^2s^2-h$ with $\lambda_T\not=0$ constant, so up to dividing $h_T$ by $s^2$, the function determining the induced metric is the effective potential studied by Cederbaum--Jahns--Vi\v{c}\'{a}nek Mart\'{i}nez \cite{cedoliviasophia} in order the characterize solutions of the ODE system (away from singular radii where these coordinates break down).
	\end{bem}
\section{Characterization of STCMC surfaces in spacelike totally umbilic warped product graphs}\label{sec_spacelikesetting}
	In this section, we will always assume that $(M_T,g^T,K^T)$ is a spacelike warped product graph in Class $\mathfrak{H}$. Assuming the (NEC), we show that we can extend the characterization of STCMC surfaces in the time-symmetric $t=\operatorname{const.}$-slices\footnote{In time symmetry we have that $\mathcal{H}^2=H^2$, so STCMC surfaces are CMC surfaces and the characterization directly follows from \cite{brendle, borghinifogagnolopinamonti}} to totally umbilic warped product graphs by applying the Alexandrov Theorem \cite{brendle, borghinifogagnolopinamonti}.
\subsection{The (NEC) on spacelike warped product graphs}\,\newline
	We first show that we can rewrite the (NEC) along any spacelike warped product graph $(M_T,g^T,K^T)$ as a tensor inequality adapted to the slice. Recall that the (NEC) along $(M_T,g^T,K^T)$ equivalently implies that for any unit vector $V_T$ on $(M_T,g^T,K^T)$ 
	\begin{align}
	\begin{split}\label{NECslice}
	0\le&\,\mu_T-\frac{1}{2}\mathfrak{R}\pm 2J_T(V_T)+\mathfrak{Ric}(V_T,V_T)\\
	=&\,\mu_T-\frac{1}{2}\mathfrak{R}\pm 2J_T(V_T)+(\tr_TK^TK^T-(K^T)^2)(V_T,V_T)\\
	&+\Ric^T(V_T,V_T)-\mathfrak{Rm}(V_T,\vec{n}^T,V_T,\vec{n}^T),
	\end{split}
	\end{align}
	where we used the contracted Gauss equation in the last line. Here $\mu_T$, $J_T$ denote the energy and momentum density of $(M_T,g^T,K^T)$, respectively, and $\tr_T$ denotes the trace with respect to $g^T$. We will similarly denote the respective quantities on the $\{t=0\}$ time slice $M_0$ with a subscript $0$. We refer to Appendix A, where we collect and derive the well-known curvature quantities for warped product graphs and spacetimes of Class $\mathfrak{H}$ for the convenience of the reader.
	
	We further define an isomorphism between the tangent bundles of $M_T$ and $M_0$ in the following way: Let $V_T=c_1f_T(s)\partial_s+\frac{c_2}{s}X$ be a tangent vector field along $M_T$, where $X$ is a unit vector field tangent in $(\mathcal{N},g_{\mathcal{N}})$. We define the vector field $V_0$ tangent to $M_0$ as $V_0\definedas c_1f(s)\partial_r+\frac{c_2}{s}X$. Note that this isomorphism is not induced by an isometry between $M_T$ and $M_0$ unless $T=\operatorname{const.}$. Together with Lemma \ref{lemma_difference}, we establish the following:
	\begin{lem}\label{lem_constraintsWarpedproductgraphs}
		For $f_0\definedas f=\sqrt{h}$ as above, we find
		\begin{enumerate}
			\item[\emph{(i)}] $\mu_T=\mu_0=\frac{1}{2}R_0=\frac{1}{2}\mathfrak{R}+\frac{\Delta_0f_0}{f_0}$,
			\item[\emph{(ii)}] $J_T\equiv0$,
			\item[\emph{(iii)}] $\Ric^T(V_T,V_T)+(\tr_TK^TK^T-(K^T)^2)(V_T,V_T)=\Ric^0(V_0,V_0)$,
			\item[\emph{(iv)}] $\mathfrak{Rm}(V_T,\vec{n}^T,V_T,\vec{n}^T)=\mathfrak{Rm}(V_0,\partial_t,V_0,\partial_t)=\frac{\Hess_0f_0(V_0,V_0)}{f_0}$
		\end{enumerate}
	\end{lem}
	\begin{bem}
		As the spacetime is static, it is unsurprising that with the above identities at hand Equation \eqref{NECslice} directly reduces to \Cref{lemma_NEC} (ii). However, as we aim to employ the result of Brendle \cite{brendle} directly on the slice, we will instead rewrite \eqref{NECslice} as a tensor inequality involving $f_T:=\sqrt{h_T}$.
		
		We also want to emphasize the vanishing momentum constraint $J_T\equiv0$, as the converse is also true in the following sense: If $(\widetilde{M},\widetilde{g},\widetilde{K})$ is an initial data set with warped product structure as above determined by the functions $\widetilde{h}$, $\widetilde{a}$, $\widetilde{b}$, then $\widetilde{J}\equiv 0$ if and only if $(\widetilde{M},\widetilde{g},\widetilde{K})$ embeds as a warped product graph into a spacetime of Class $\mathfrak{H}$ with $h:=\widetilde{h}-s^2\widetilde{b}^2$. 
		
		In \cite{cabrerawolff}, Cabrera\,Pacheco and the author construct non-time symmetric initial data sets from the constraint equations with similar warped product structure. In spherical symmetry, the above observation then allows one to further construct the maximal future development under the additional assumption $J\equiv0$.
	\end{bem}
	\begin{proof}[Proof of Lemma \ref{lem_constraintsWarpedproductgraphs}]
		In view of the Remark, we only prove (ii) and refer to the proof of (i), (iii), (iv) to Appendix \ref{appendix_curvature}. In coordinates, $J_T$ is given as
		\[
		(J_T)_i=(g^T)^{jk}\left(K_{ki,j}-\Gamma_{jk}^lK_{li}-\Gamma_{ij}^lK_{kl}\right)-\tr_TK_{,i},
		\]
		for indices $i,j,k,l\in \{s,I,J,K,L\}$. Using the block diagonal structure of $K$ and the well-known identities for the Christoffel symbols in class $\mathfrak{H}$, a direct computation shows that
		\[
		J^T=\frac{2}{s}\left(h_Ta_T-b_T-sb_T'\right)\d s.
		\]
		Thus, it remains to show that $h_Ta_T-b_T-sb_T'$ vanishes along $M_T$. As $b_T$ is in particular continuous, there exists a closed set $\mathcal{X}\subset I_T:=(r_1,r_2)$ of measure zero, such that $ I_T\setminus \mathcal{X}$ lies dense in $I_T$, and for all $s\in I_T\setminus \mathcal{X}$ there exists an open neighborhood $U_s$ of $s$, such that either $b_T\not=0$ or $b_T$ vanishes identically on $U_s$. In the first case, multiplying the equation by $b_T$ yields
		\[
		b_T\left(h_Ta_T-b_T-sb_T'\right)=h_Ta_T-b_T^2-sb_T'b_T=0,
		\]
		which vanishes by the identities of Lemma \ref{lemma_difference} and using
		\[
		2b_T'b_T=(b_T^2)'=\frac{h_T'-h'}{s^2}-2\frac{h_T-h}{s^3}.
		\]
		Since we assumed $b_T\not=0$ , we have $h_Ta_T-b_T-sb_T'=0$. On the other hand, if $b_T$ vanishes identically on a neighborhood, then $T=\operatorname{const.}$ by Remark \ref{rem_destriptionbydifference}, so $h_T=h$ and $a=0$. In particular $h_Ta_T-b_T-sb_T'=0$. Therefore $h_Ta_T-b_T-sb_T'$ vanishes on $I_T\setminus\mathcal{X}$. By continuity, it has to vanish on all of $I_T$, which concludes the proof of (ii).
	\end{proof}
	We now establish the relevant tensor inequality adapted to the  warped product graphs  $(M_T,g^T,K^T)$.
	\begin{prop}\label{prop_mainineq}
		Let $(\mathfrak{M},\mathfrak{g})$ be a spacetime of Class $\mathfrak{H}$ that satisfies the (NEC), and let $(M_T,g^T,K^T)$ be a warped product graph in $(\mathfrak{M},\mathfrak{g})$. Then, for all unit vector fields $V_T$ on $(M_T,g^T,K^T)$ we have
		\begin{align}\label{inequ_main}
		B^T(V_T,V_T)\le \frac{\Delta_Tf_T}{f_T}-\frac{\Hess_Tf_T(V_T,V_T)}{f_T}+\Ric^T(V_T,V_T),
		\end{align}
		where 
		\[
		B^T:=\left(\frac{1}{2}(h_T-h)''+\frac{(n-3)}{2s}(h_T-h)'-\frac{(n-2)}{s^2}(h_T-h)\right)s^2g_\mathcal{N}.
		\]
	\end{prop}
	\begin{proof}
		Note that
		\begin{align}
		\begin{split}\label{eq_K_T_1}
		(K^T)^2&=h_Ta_T^2\d s^2+b_T^2s^2g_{\mathcal{N}},\\
		\tr_{M_T}K^T&=h_Ta_T+(n-1)b_T,\\
		\vert K^T\vert^2&=h_T^2a_T^2+(n-1)b_T^2,
		\end{split}
		\end{align}
		and using Lemma \ref{lemma_difference}, direct computation shows
		\begin{align}\label{eq_K_T_2}
		\tr_{M_T}K^TK^T-(K^T)^2=\frac{(n-1)}{2}\frac{h_T'-h'}{sh_T}\d s^2+\left(\frac{h_T'-h'}{2s}+\frac{(n-2)}{s^2}(h_T-h)\right)s^2g_{\mathcal{N}}.
		\end{align}
		For a unit vector $V_T=c_1f_T\partial_s+\frac{c_2}{s}X$ on $M_T$, the (NEC) gives
		\[
		0\le \frac{\Delta_0f_0}{f_0}-\frac{\Hess_0f_0}{f_0}(V_0,V_0)+\Ric^T(V_T,V_T)+\left(\tr_T K^T K^T-(K^T)^2\right)(V_T,V_T)
		\]
		by Lemma \ref{lem_constraintsWarpedproductgraphs}.
		Note that a direct computation yields
		\begin{align*}
		&\left(\frac{\Delta_Tf_T}{f_T}-\frac{\Delta_0f_0}{f_0}\right)-\left(\frac{\Hess_Tf(V_T,V_T)}{f_T}-\frac{\Hess_0f_T(V_0,V_0)}{f_0}\right)-\left(\tr_{M_T}K^TK^T-K^2\right)(V_T,V_T)\\
		&=c_2^2\left(\frac{1}{2}(h_T-h)''+\frac{(n-3)}{2s}(h_T-h)'-\frac{(n-2)}{s^2}(h_T-h)\right)\\
		&=B^T(V_T,V_T).
		\end{align*}
		Inserting this into the tensor inequality above yields the claim.
	\end{proof}
	Provided that $B^T$ is positive semi-definite, the (NEC) on $(\mathfrak{M},\mathfrak{g})$ in particular implies the desired tensor inequality
	\begin{align}\label{eq_tensorineqslice}
	\frac{\Delta_Tf_T}{f_T}-\frac{\Hess_Tf_T(V,V)}{f_T}+\Ric^T(V,V)\ge 0
	\end{align}
	on general warped product initial data sets $(M_T,g^T,K^T)$. Moreover, $B^T\ge 0$ if and only if $x:=h_T-h$ is a non-negative function satisfying the linear ordinary differential inequality
	\begin{align}\label{mainODI}
	\frac{1}{2}x''+\frac{(n-3)}{2s}x'-\frac{(n-2)}{s^2}x\ge 0.
	\end{align}
	Note that this is the same differential inequality as in Lemma \ref{lemma_NEC} (iv) for the function $h-\alpha$, which is equivalent to the (NEC) on $(\mathfrak{M},\mathfrak{g})$. By linearity, we have that $h_T-\alpha$ solves the above differential inequality \eqref{mainODI}, which by Lemma \ref{lemma_NEC} implies that the spacetime of Class $\mathfrak{H}$ with metric coefficient $h_T$ (and the same fibre $\mathcal{N}$) satisfies the (NEC). 
	\begin{bem}\label{bem_NEC}
		The exact solutions of \eqref{mainODI} as an ODE are given by a $2$-parameter family of solutions of the form
		\[
		x=\frac{C_1}{s^{-n+2}}+C_2s^2.
		\]
		In spherical symmetry, $h=1+x$ correspond to the Schwarzschild de\,Sitter and Schwarzschild anti-de\,Sitter family depending on the sign of $C_2$, which describe the static, spherically symmetric Vacuum solutions of the Einstein Equations (with cosmological constant depending on $C_2$). These are precisely the spacetimes of class $\mathcal{S}$ such that the time-symmetric slices have constant scalar curvature. Compare \cite[Lemma 2.3]{corvino}.
	\end{bem}
	In this sense, if $B^T\ge0$, the spacetime of Class $\mathfrak{H}$ with metric coefficient $h_T$ inherits the (NEC) from the spacetime of Class $\mathfrak{H}$ with metric coefficient $h$. This makes apparent that we can use the (NEC) on $(\mathfrak{M},\mathfrak{g})$ to classify CMC surfaces on a large class of general warped product graphs $(M_T,g^T,K^T)$ provided we can extend the tensor inequality in an appropriate way until the first zero $r_T$ of $h_T$ to verify Brendle's condition (H1) in \cite{brendle}. As $h_T\ge h$ on $I$, a minimal, inner boundary of $(M_T,g^T)$ will in general be hidden behind a Killing horizon of $(\mathfrak{M},\mathfrak{g})$.
\subsection{Extending the graphs past the Killing horizon}\label{sec_extentsion}\,\newline
	We recall that Killing horizons arise in spacetimes of Class $\mathfrak{H}$ as zeros of $h$. We will always assume from now on that $h$ has finitely many positive, simple zeroes \[
	{r_0:=0<r_1<\dotsc<r_N=r_H<\infty=:r_{N+1}},
	\]
	such that all arising Killing horizons are non-degenerate, i.e., $h'(r_l)\not=0$ for all $1\le l\le N$. As (H2) is equivalent to the fact that the outermost Killing horizon $\{r=r_H\}$ is non-degenerate, this is rather a necessary than a restrictive assumption in lieu of applying the results of Brendle~\cite{brendle} and Borghini--Fogagnolo--Pinamonti \cite{borghinifogagnolopinamonti}.
	
	However, as $r\to r_N$ the $(t,r)$-coordinate system breaks down, and in general $M_T$ can no longer be described as a graph of a radial function $T$, since by Remark \ref{rem_destriptionbydifference} 
	\[
	\btr{T'}=\frac{1}{h}\sqrt{\frac{h_T-h}{h_T}}.
	\]
	We want to argue that we can extend the graph of $T$ past any Killing horizon $\{r=r_l\}$ for all $1\le l\le N$ with $h_T(r_l)>0$ in different coordinates and extend the notion of the (NEC) as an ordinary differential inequality on a suitable spacetime extension. As we only need this up to the inner boundary of $M_T$, i.e., $\{s=r_T\}$ for the largest zero $r_T$ of $h_T$, we always have that ${h_T(r_l)>h(r_l)=0}$ for all $1\le l\le N$ with $r_l>r_T$.
	
	Using the above assumptions that each Killing horizon is non-degenerate, Brill--Hayward \cite{brillhayward}, Schindler--Aguirre \cite{schindagui}, and Cederbaum and the author \cite{cedwolff} showed independently that a spacetime $(\mathfrak{M},\mathfrak{g})$ of Class $\mathfrak{H}$ admits a spacetime extension called the \emph{generalized Kruskal--Szekeres extension} under the above assumptions, that extends the radial coordinate $r$ to $(0,\infty)$ across the zeros of $h$. Throughout this section we will use similar conventions as in \cite{cedwolff}. In each coordinate chart, the spacetime extension is given by a warped product manifold $(\mathbb{P}_l\times \mathcal{N},\widetilde{g}_l)$, where
	\begin{align*}
	\mathbb{P}_l&\definedas\{(u,v)\in\R^2:uv\in \text{Im}(\Phi_l)\},\\
	\widetilde{g}_l&=(F_l\circ\rho)(\d u\otimes \d v+\d v\otimes \d u)+\rho^2g_{\mathcal{N}},
	\end{align*}
	with $\rho=\Phi_l^{-1}(uv)$, $F_l=\frac{2C_l}{\Phi_l'}$, where $\Phi_l$ is the unique strictly increasing solution of 
	\begin{align}\label{ODEKruskal}
	\frac{\Phi_l}{\Phi_l'}=C_lh,
	\end{align}
	on $(r_{l-1},r_{l+
		1})$ with $\Phi'(r_l)=1$, $C_l\definedas\frac{1}{h'(r_l)}$, $1\le l\le N$. Note that if $h>0$ on $(r_N,\infty)$, the original spacetime $(\mathfrak{M},\mathfrak{g})$ corresponds to $\{u,v>0\}$ in $\mathbb{P}_N\times \mathcal{N}$. Moreover, we have the explicit coordinate transformations between $(u,v)$ and $(t,r)$ coordinates
	\begin{align*}
	v(t,r)=\sqrt{\Phi_l(r)}\exp\left(\frac{1}{2C_l}t\right),\\
	u(t,r)=\sqrt{\Phi_l(r)}\exp\left(-\frac{1}{2C_l}t\right)
	\end{align*}
	on each coordinate patch $\R\times(r_j,r_{j+1})\times \mathcal{N}$ ($j\in\{l-1,l\}$), where $(t,r)$-coordinates are defined.
	A direct computation in the $(u,v)$-coordinates using Equation \eqref{ODEKruskal} gives
	\begin{align*}
	\mathfrak{Ric}_{vv}&=\mathfrak{Ric}_{uu}=0,\\
	\mathfrak{Ric}_{uv}&=-\frac{F_l}{2}\left(h''+\frac{(n-1)}{\rho}h'\right),\\
	\mathfrak{Ric}_{IJ}&=(\Ric_{g_\mathcal{N}})_{IJ}-\left((n-2)h+\rho h'\right)(g_\mathcal{N})_{IJ},
	\end{align*}
	cf. \cite[Proposition B.1]{cedwolff}.
	Let $L=a\partial_u+b\partial_v+\frac{c}{\rho}X$ be a null vector field, where $X$ is a unit vector in $(\mathcal{N},g_{\mathcal{N}})$. Then
	\[
	0=2F_lab+c^2.
	\]
	Using this identity, we see that
	\begin{align*}
	0&\le \mathfrak{Ric}(L,L)\\
	&=-F_lab\left(h''+\frac{(n-1)}{\rho}h'\right)+\frac{c^2}{\rho^2}\left(\Ric_{g_\mathcal{N}}(X,X)-(n-2)h+\rho h'\right)\\
	&=c^2\left(\frac{1}{2}(h-1)''+\frac{(n-3)}{2\rho}(h-1)'-\frac{(n-2)}{\rho^2}(h)+\Ric_{g_\mathcal{N}}(X,X)\right),
	\end{align*}
	which is again equivalent to the linear ODI \eqref{mainODI} for $x=h-\alpha$ by the same arguments as in the proof of Lemma \ref{lemma_NEC}. Therefore the spacetime extension satisfies the (NEC) if and only if $x=h-\alpha$ satisfies \eqref{mainODI} on $(0,\infty)$.
	
	A crucial observation in \cite{cedwolff} is that any solution $\Phi_l$ of \eqref{ODEKruskal} is of the form 
	\[
	\Phi_l=h\exp\left(\frac{R_l}{C_l}\right),
	\]
	where $R_l$ is a smooth function on $(r_{l-1},r_{l+1})$ uniquely determined up to a constant. Although the construction circumvents the need to do so, this yields a tortoise function $R^*$, i.e., a primitive of $\frac{1}{h}$, a-posteriori on each of the intervals $(r_{l-1},r_l)$ and $(r_l,r_{l+1})$ by defining $R^*\definedas C_l\ln(\btr{\Phi_l})$.
	
	Now, we notice that by Remark \ref{rem_destriptionbydifference} $T$ is a primitive of the function $\pm\frac{1}{\eta_T}$, where\linebreak ${\eta_T:=h\sqrt{\frac{h_T}{h_T-h}}}$ is well-defined and $\eta_T'(r_l)=h'(r_l)\not=0$ for any $r_l>r_T$. In particular, the construction of Cederbaum and the author \cite{cedwolff} yields that $T$ satisfies
	\[
	T=\pm\left(C_l\ln(\btr{h})+C_l\frac{1}{2}\ln\left(\frac{h_T}{h_T-h}\right)+\widetilde{R}_T\right)
	\]
	for some smooth function $\widetilde{R}_T$  on $(r_{l-1},r_{l+1})$. Using this explicit behavior of $T$, we see that in the $+$-case
	\begin{align*}
	v(s)=h\sqrt[4]{\frac{h_T-h}{h_T}}\exp\left(\frac{1}{2K}(R+\widetilde{R}_T)\right),\\
	u(s)=\sqrt[4]{\frac{h_T}{h_T-h}}\exp\left(\frac{1}{2K}(R-\widetilde{R}_T)\right),
	\end{align*}
	so $M_T$ extends smoothly across the horizon and crosses the horizon at $v=0$, $u=u(r_l)$, and similarly at $v=v(r_l)$, $u=0$ in the $-$-case. Therefore we can extend any warped product graph across any non-degenerate Killing horizon up to its minimal, inner boundary.\newline
\subsection{STCMC characterization on totally umbilic warped product graphs}\label{sec_maintheorem}\,\newline
	Combining the result of the previous subsections with the results of Brendle~\cite{brendle} and Borghini--Fogagnolo--Pinamonti \cite{borghinifogagnolopinamonti}, and of Cederbaum and the author \cite{cedwolff}, we acquire the following:
	\begin{thm}\label{thm_main1}
		Let $h:(0,\infty)\to\R$ be a smooth function with finitely many, positive simple zeroes $r_1<\dotsc< r_N$, $(\mathcal{N},g_{\mathcal{N}})$ an $(n-1)$-dimensional Riemannian manifold ($n\ge 2$). Let $(\mathfrak{M},\mathfrak{g})$ be the corresponding spacetime of Class $\mathfrak{H}$ with metric coefficient $h$ and fibre $\mathcal{N}$, and assume that the generalized Kruskal--Szekeres extension of $(\mathfrak{M},\mathfrak{g})$ satisfies the (NEC) condition.\newline
		Let $x$ be some non-negative function satisfying \eqref{mainODI}, and consider the warped product graph $M_T$, where $T$ is such that $h_T=h+x$. Then $M_T$ extends into the generalized Kruskal--Szekeres extension until its minimal inner boundary, corresponding to the first zero $r_T$ of $h_T$.\newline
		Additionally, if $r_T>0$ with $h_T'(r_T)>0$ then all compact CMC surfaces in $M_T$ are leaves $\{s\}\times \mathcal{N}$.
	\end{thm}
	\begin{proof}
		By the results in \cite{cedwolff} as summarized in Subsection \ref{sec_extentsion}, $(\mathfrak{M},\mathfrak{g})$ extends onto all positive radii, and the generalized Kruskal--Szekeres extension is covered by a countable, smooth atlas. As observed in the previous subsection the (NEC) implies that the ODI \eqref{mainODI} holds for $h-\alpha$ on all of $(0,\infty)$, and the graph $M_T$ of $T$ with $h_T=h+x$ is well defined across any non-degenerate Killing horizon up until the first zero $r_T$ of $h_T$.
		
		Since $x$ is a non-negative solution of \eqref{mainODI}, the tensor inequality \eqref{eq_tensorineqslice} holds, and by Remark \ref{rem_destriptionbydifference} there exists a graph $T$ such that $h_T=h+x$, where $T$ is uniquely determined by $x$ up to a choice of sign of $T'$ and a constant of integration. Now assume that the first zero $r_T$ of $h_T$ satisfies $h_T'(r_T)>0$. In particular, conditions (H1)-(H3) in \cite[Theorem 1.1]{brendle} are satisfied. Thus, any compact CMC surface in $M_T$ is totally umbilic. Moreover, by \cite[Corollary 1.3]{borghinifogagnolopinamonti} any compact CMC surface is in fact a leaf of the canonical foliation $\{s\}\times \mathcal{N}$. 
	\end{proof}
	Note that this result is independent of the extrinsic curvature $K$. Therefore it also suffices to the results in \cite{brendle, borghinifogagnolopinamonti} directly to the totally geodesic slices in the spacetime of Class $\mathfrak{H}$ with metric coefficient $h_T$, as this spacetime will satisfy the (NEC) by Remark \ref{bem_NEC}. Note that this observation is consistent with the duality of constant, positive mean curvature slices in spacetimes with zero cosmological constant and maximal slices in spacetimes with negative cosmological constant, cf. \cite{chruscieltod}. 
	
	We now want to incorporate the extrinsic curvature $K$ into our result. Due to the difficulty in adapting the methods in \cite{brendle, borghinifogagnolopinamonti} in the presence of $P=\tr_\Sigma K$ and its evolution, we restrict ourselves to the special case of totally umbilic warped product graphs, which we have fully characterized in Corollary \ref{kor_hyperboloids}.
	Note that on a hyperboloid as defined in Section \ref{sec_graoh}, $P=\tr_\Sigma K=(n-1)\lambda$ is constant and the same for any embedded surface $\Sigma$, so the evolution of $P$ along any deformation is trivial. Hence, any surface $\Sigma$ in $(M_T,g^T,K^T)$ has constant spacetime mean curvature $\mathcal{H}^2$, and constant expansion $\theta_\pm$, if and only if it is a CMC surface.
	Moreover, by Remark \ref{bem_NEC} $Cs^2$ is an exact solution of \eqref{mainODI}, so it in fact solves \eqref{mainODI} as an ODE. Using Theorem \ref{thm_main1}, we acquire our main result for totally umbilic warped product graphs in Class $\mathfrak{H}$.
	\begin{thm}\label{thm_main2}
		Let $h:(0,\infty)\to\R$ be a smooth function with finitely many, positive simple zeroes $r_1<\dotsc< r_N$, $(\mathcal{N},g_{\mathcal{N}})$ an $(n-1)$-dimensional Riemannian manifold ($n\ge 2$). Let $(\mathfrak{M},\mathfrak{g})$ be the corresponding spacetime of Class $\mathfrak{H}$ with metric coefficient $h$ and fibre $\mathcal{N}$. Assume that the generalized Kruskal--Szekeres extension satisfies the (NEC) condition.\newline
		Then there exists a constant $C_0=C_0(h,h')\in(0,\infty]$ such that for any hyperboloid \linebreak $(M_T,g^T,K^T)$ with umbilicity factor $\lambda_T$ satisfying $\lambda_T^2<C_0$ we have:
		If $\Sigma\subset M_T$ is an orientable, closed, embedded hyper\-surface with constant spacetime mean curvature, then $\Sigma$ is a leaf $\{s\}\times \mathcal{N}$.
	\end{thm}
	\begin{proof}
		For $C=\lambda_T^2$ small enough, $h_T=h+Cs^2$ has at least one positive zero $r_T$ with $r_T\to r_N$ as $C\to0$. By continuity, we have $h'(q_T)>0$ for small enough $C$. We define $C_0$ as the supremum over all $C$, such that these conditions are still satisfied.
		Thus, Theorem \ref{thm_main1} applies to a hyperboloid $(M_T,g^T,K^T)$ satisfying $\lambda_T^2<C_0$, and any CMC surface $\Sigma$ is a leaf $\{s\}\times \mathcal{N}$. Since $(M_T,g^T,K^T)$ has constant umbilicity factor, any STCMC surface is a leaf $\{s\}\times \mathcal{N}$.
	\end{proof}
	Note that all assumptions are in particular satisfied for any constant $C\ge 0$ in the Kruskal--Szekeres extension of the Schwarzschild spacetime with positive mass $m>0$ corresponding to 
	\[
		h(r)=1-\frac{2m}{r^{n-2}}
	\]
	in spherical symmetry.
	\begin{kor}
		Let $(\mathfrak{M},\mathfrak{g})$ be the Schwarzschild spacetime with positive mass. Then any closed, embedded STCMC surface $\Sigma$ in an hyperboloid $(M_T,g^T,K^T)$ is a slice $\{s\}\times \mathbb{S}^{n-1}$.
	\end{kor}
	\begin{bem}\label{bem_STCMCspheres}
		Note that a direct computation yields that
		\[
			\mathcal{H}^2=\frac{(n-1)h(s)}{s^2}
		\]
		for a leaf $\{s\}\times \mathcal{N}$ for any warped product graph $(M_T,g^T,K^T)$. As we have to extend any hyperboloid with $\lambda_T\not=0$ across the horizon where $h(r_H)=0$ into a region where $h<0$, both the case of generalized apparent horizons $\mathcal{H}^2=0$, and trapped STCMC surfaces with $\mathcal{H}^2<0$ naturally occur in hyperboloids.
	\end{bem}
\section{Characterizing photon surfaces with constant umbilicity factor}\label{sec_photonsurfaces}

	Arguing in analogue to Brendle \cite{brendle} that umbilic CMC surfaces are canonical leaves in the Riemannian setting under condition (H4), we obtain a short proof for a characterization of photon surfaces with constant umbilicity factor $\lambda$ imposing a condition on the eigenvalues of the Ricci tensor in the spacetime setting in Class $\mathfrak{H}$ (and the respective Kruskal--Szekeres extension). Although utilizing completely different methods, the statement is similar to a result by Cederbaum--Galloway {\cite[Theorem 3.8]{cedgal}} in spherical symmetry under the additional assumption that the umbilicity factor is constant.
\subsection{A spacetime Ricci eigenvalue argument}\,\newline
	Recall that a photon surface $\mathcal{P}^n$ is a smooth, timelike, totally umbilic hypersurface in an $(n+1)$-dimensional spacetime $(\mathfrak{M},\mathfrak{g})$, so denoting the induced metric on $\mathcal{P}^n$ by $\mathfrak{p}$ and its second fundamental form by $\mathfrak{h}$, we have 
	\[
	\mathfrak{h}=\lambda\mathfrak{p}
	\] 
	for some smooth function $\lambda$ on $\mathcal{P}^n$. The name is due to the fact that a timelike hypersurface is totally umbilic if and only if null geodesics starting tangent to the hypersurface must remain tangent to $\mathcal{P}^n$ \cite{claudelvirbhadraellis, perlick}. We refer the interested reader to \cite{cedgal} for a more complete introduction and list of references.\newline
	In this section, we will consider photon surfaces $\mathcal{P}^n$ in spacetimes of Class $\mathfrak{H}$ and their respective generalized Kruskal--Szekeres extension as defined in Subsection \ref{sec_extentsion}. For  photon surfaces, the Codazzi equation implies
	\begin{align}\label{eq_Codazzieq}
		(n-1)\nabla^n_V\lambda=-\mathfrak{Ric}(Y,\eta)
	\end{align} 
	for any tangent vector field $Y\in\Gamma(T\mathcal{P}^n)$, where $\nabla^n$ denotes the exterior derivative on $\mathcal{P}^n$ and $\eta$ the spacelike unit normal to $\mathcal{P}^n$, respectively. In spacetimes of Class $\mathfrak{H}$, the Ricci curvature tensor has an eigenvalue $\beta$ with a corresponding eigenspace that is at least $2$-dimensional and contains $\{\partial_t,\partial_r\}$. More precisely
	\begin{align*}
	\beta&=-\frac{1}{2}\left(h''+\frac{(n-1)}{r}h'\right).
	\end{align*}
	Since $\mathfrak{Ric}_{IJ}=\left(\Ric_{g_\mathcal{N}}\right)_{IJ}-\left((n-2)h+rh'\right)\left(g_{\mathcal{N}}\right)_{IJ}$, all other eigenvalues of $\mathfrak{Ric}$ are characterized by the eigenvalues of $\Ric_{g_\mathcal{N}}$. To ensure that the corresponding eigenspace of $\beta$ is exactly $2$-dimensional and thus spanned by $\{\partial_t,\partial_r\}$, we require that the difference between $\beta$ and any eigenvalue of $\mathfrak{Ric}\vert_{T\mathcal{N}\times T\mathcal{N}}$ is non-trivial. Equivalently, 
	\begin{align}\label{eq_eigenvalueTRUE2}
	\frac{1}{2}h''+\frac{(n-3)}{2r}h'-\frac{n-2}{r^2}h+r^{-2}\Ric_{g_\mathcal{N}}(X,X)\not= 0
	\end{align}
	for any unit tangent vector field $X$ in $(\mathcal{N},g_{\mathcal{N}})$\footnote{Is suffices to assume this for any unit tangent eigenvector $X$ of $\Ric_{g_\mathcal{N}}$}. Note that if \eqref{eq_eigenvalueTRUE2} is strictly positive, then $\beta$ is the smallest eigenvalue of $\mathfrak{Ric}$ and $(\mathfrak{M},\mathfrak{g})$ satisfies the (NEC) by Lemma \ref{lemma_NEC}. Conversely, the (NEC) implies that $\beta$ is the smallest eigenvalue of $\mathfrak{Ric}$. However, we need the additional assumption that the (NEC) is strict for any null vector field $L$ that is not perpendicular to $\mathcal{N}$ to conclude that \eqref{eq_eigenvalueTRUE2} is indeed satisfied. Recall from Remark \ref{bem_NEC} that in spherical symmetry equality implies $h=1+\frac{C_1}{r^{n-2}}+C_2r^2$, so \eqref{eq_eigenvalueTRUE2} in Class $\mathfrak{H}$ is generally satisfied in spherical symmetry outside of a dense subset unless $h$ locally corresponds to the Schwarzschild de\,Sitter/Schwarzschild anti de\,Sitter family. 
	All of the above observations naturally extend to the generalized Kruskal--Szekeres extension, and in this case \eqref{eq_eigenvalueTRUE2} on $(0,\infty)$ is equivalent to the fact, that the corresponding eigenspace of $\beta$ is exactly $2$-dimensional and spanned by $\{\partial_u,\partial_v\}$.\newline
	Using the Codazzi equation as the main tool we now prove a characterization of photon surfaces with constant umbilicity factor in a spacetime of Class $\mathfrak{H}$ or its generalized Kruskal--Szekeres extension satisfying \eqref{eq_eigenvalueTRUE2}.
	\begin{thm}\label{thm_charakterisationphotonsurfaces}
		Let $\mathcal{P}^n$ be a connected photon surface with constant umbilicity factor $\lambda$ in a spacetime of Class $\mathfrak{H}$ or its respective generalized Kruskal--Szekeres extension. Assume further that \eqref{eq_eigenvalueTRUE2} is satisfied on a dense set of radii in $(0,\infty)$. Then $\mathcal{P}^n$ is either symmetric or $\mathcal{P}^n$ is totally geodesic with parallel unit normal vector $\eta$ everywhere tangent to $\mathcal{N}$.
	\end{thm}\newpage
	\begin{bem}
		In this context, we understand a photon surface to be symmetric in a spacetime of Class $\mathfrak{H}$ or its generalized Kruskal--Szekeres extension if for every point $p\in\mathcal{P}^n$ the canonical lift of the tangent space of $\mathcal{N}$ at $p$ is a subspace of the tangent space $T_p\mathcal{P}^n$ of $\mathcal{P}^n$, and the unit normal $\eta$ is everywhere spanned by $\{\partial_t,\partial_r\}$ or $\{\partial_u,\partial_v\}$ respectively. This notion agrees with the definition of symmetric photon surfaces in \cite{cedwolff}, and Cederbaum--Galloway \cite{cedgal}, Cederbaum--Jahns--Vi\v{c}\'{a}nek Mart\'{i}nez \cite{cedoliviasophia} in spherical symmetry via a profile curve $\gamma$ mapping into the $(t,r)$- or $(u,v)$-coordinate plane, respectively.
		
		Assuming to be in the exact setting of Cederbaum--Galloway in \cite{cedgal}, we note that a non-empty intersection between a $\{t=\operatorname{const.}\}$ slice and a totally geodesic photon surface with parallel unit normal vector $\eta$ everywhere tangent to $\mathcal{N}$, is (a piece of) a centered hyperplane in isotropic coordinates as defined in \cite{cedgal}. In particular, Theorem \ref{thm_charakterisationphotonsurfaces} draws the same conclusion in this setting as the result by Cederbaum--Galloway \cite[Theorem 3.8]{cedgal} under the additional assumption of $\lambda=\operatorname{const.}$. We refer to the next subsection for a comparison of the two statements.
	\end{bem}
	\begin{proof}
		Let $\mathcal{P}^n$ be a photon surface with constant umbilicity factor $\lambda$. Then the Codazzi equation implies 
		\[
		\mathfrak{Ric}(Y,\eta)=-(n-1)\nabla_V\lambda=0,
		\]
		for any tangent vector $Y$. Therefore $\eta$ is an eigenvector of $\mathfrak{Ric}$ along $\mathcal{P}^n$. Since \eqref{eq_eigenvalueTRUE2} is satisfied on a dense set of radii in $(0,\infty)$, continuity yields that \eqref{eq_eigenvalueTRUE2} is in fact satisfied on $(0,\infty)\setminus\mathcal{X}$, where $\mathcal{X}$ is a closed set that has measure zero. By assumption the eigenspace of the eigenvalue $\beta$ of $\mathfrak{Ric}$ is spanned by $\{\partial_t,\partial_r\}$ ($\{\partial_u,\partial_v\})$ away from $\mathcal{X}$. If $\mathcal{P}^n\subset \R\times \mathcal{X}\times \mathcal{N}$, as $\mathcal{P}^n$ is connected, $\mathcal{P}^n\subset\R\times \{r=r_0\}\times \mathcal{N}$ for some $r_0\in \mathcal{X}$, in particular $\mathcal{P}^n$ is rotationally symmetric.\newline
		Now assume $\mathcal{P}^n\subset\R\times I\times \mathcal{N}$ for some open interval $I$ with $I\cap \mathcal{X}=\emptyset$. Then \eqref{eq_eigenvalueTRUE2} holds on $I$, so we have $\eta\in\operatorname{span}(\partial_t,\partial_r)$ (or $\eta\in\operatorname{span}(\partial_u,\partial_v)$ in the case of the generalized Kruskal--Szekeres extension) or $\eta$ is everywhere tangent to $\mathcal{N}$. In the first case $\eta$ is everywhere perpendicular to $\mathcal{N}$ and hence everywhere a lift of $T\mathcal{N}$ has to be a subspace of the tangent bundle of $\mathcal{P}^n$, i.e., $\mathcal{P}^n$ is rotationally symmetric. If $\eta$ is everywhere tangent to $\mathcal{N}$, then away from any Killing horizon (which $\mathcal{P}^n$ can only cross at a discrete set which we can also exclude from $I$) $h\not=0$, and $\partial_t$ is well-defined and tangent to $\mathcal{P}^n$. Staticity then implies
		\[
		-\lambda h=\lambda\mathfrak{p}(\partial_t,\partial_t)=\mathfrak{h}(\partial_t,\partial_t)=-g(\nabla_{\partial_t}\partial_t,\eta)=0,
		\]
		so $\lambda=0$. Hence $\mathfrak{h}$ vanishes identically and $\mathcal{P}^n$ is totally geodesic. In particular, $\eta$ is parallel along $\mathcal{P}^n$.
		
		Lastly, if $I\cap \mathcal{X}\not=\emptyset$, then the above argument holds on any connected component of $\mathcal{P}^n\cap (\R\times(I\setminus \mathcal{X})\times\mathcal{N})$, i.e., on any connected component $\eta$ is either orthogonal or tangential to $\mathcal{N}$. Assume that there is an open component $C$ such that $\eta$ is everywhere tangent to $\mathcal{N}$. In particular, $\overline{C}\subset \mathcal{P}^n$ intersects $\R\times \mathcal{X}\times \mathcal{N}$ transversally. As $\eta$ is continuous and nowhere-vanishing, this implies that $\eta$ is tangent to $\mathcal{N}$ for all open connected components of $\mathcal{P}^n\cap (\R\times(I\setminus \mathcal{X})\times\mathcal{N})$ with closure in $\mathcal{P}^n$ intersecting $\partial C$. As $\mathcal{P}^n$ is connected, applying the above argument iteratively eventually covers all of $\mathcal{P}^n$. This concludes the proof.
	\end{proof}
\subsection{Comparison to the characterization of Cederbaum--Galloway}\,\newline
	Cederbaum--Galloway characterized photon surfaces in static, spherically symmetric spacetimes in isotropic coordinates that are of the form $(\R\times D^n,-\widetilde{N}^2\d t^2+\Psi^2\delta)$, where\linebreak ${D^n\definedas\{x\in R^n\colon \norm{x}=s\in I\}}$ for some positive open Interval $I\in\R$, and $\Psi,\widetilde{N}$ are smooth, positive functions on $I$. In their assumption, they exclude spacetimes satisfying 
	\begin{align}\label{eq_spacetimeconformalyflat1}
	\frac{\widetilde{N}'}{\widetilde{N}}=\frac{\Psi'}{\Psi},
	\end{align}
	as this implies that these spacetimes are spacetime conformally flat
	\[
	-\widetilde{N}^2\d t^2+\Psi^2\delta=\Psi^2\left(-A\d t^2+\delta\right)
	\]
	for some positive constant $A$, and thus posses the same plethora of ``off-center'' photon surfaces as the Minkowski spacetime. Assuming that \eqref{eq_spacetimeconformalyflat1} does not hold, Cederbaum--Galloway showed in \cite{cedgal}, that any photon surface must necessarily be rotationally symmetric or a centered vertical hyperplane in this coordinate system. See \cite[Theorem 3.8]{cedgal} for a precise statement. A spacetime of class $\mathcal{S}$ can always be locally rewritten in isotropic coordinates by defining $s$ as a primitive of $\left(r\sqrt{h(r)}\right)^{-1}$, and setting $\Psi(s)\definedas\frac{r(s)}{s}$ and $\widetilde{N}^2(s)=h(r(s))$, where we denote the inverse of $s(r)$ by $r(s)$. On the other hand, a static, spherically symmetric spacetime in isotropic coordinates can be globally rewritten as a spacetime of class $\SCal$ if and only if 
	\begin{align}\label{eq_spacetimeconformalyflat2}
	\widetilde{N}^2=\left(1+\frac{s\Psi'}{\Psi}\right)^2,
	\end{align}
	with $r\definedas s\Psi$ and $h(r)=\widetilde{N}^2(s(r))$, where $s(r)$ denotes the inverse of $r(s)$. Thus \cite[Theorem 3.8]{cedgal} and Theorem \ref{thm_charakterisationphotonsurfaces} achieve a similar characterization of photon surfaces under different assumptions. Recall that Theorem \ref{thm_charakterisationphotonsurfaces} additionally imposes that $\lambda$ is constant, however the proof of Cederbaum--Galloway heavily relies on the conformally Euclidean structure of the time--symmetric time slices in class $\SCal$ and that any timelike hypersurface intersects them transversally. Thus, it is not immediate to extend this result to cases $\mathcal{N}\not=\mathbb{S}^{n-1}$ or past the Killing horizon in the generalized Kruskal--Szekeres extension as the totally geodesic $\{t=\operatorname{const.}\}$-slices are timelike in a region with $h<0$, and hence are photon surfaces themselves.
	
	However, even within class $\SCal$ the assumptions of both theorems appear to be distinctly different. On the one hand, the Schwarzschild spacetime satisfies the assumptions of Cederbaum--Galloway, but can not satisfy \eqref{eq_eigenvalueTRUE2} as a vacuum solution. Conversely, the following example shows that there is a $1$-parameter family of spacetimes within class $\mathcal{S}$ that satisfy \eqref{eq_spacetimeconformalyflat1} locally, but globally allow for a characterization of photon surfaces with constant umbilicity factor using Theorem \ref{thm_charakterisationphotonsurfaces}.
	
	Combining \eqref{eq_spacetimeconformalyflat1} and \eqref{eq_spacetimeconformalyflat2}, we see that this question reduces to an ODE on $\Psi$, which we can explicitly solve for. The solutions belong to a $2$-parameter family and are given by
	\[
	\Psi(s)=\frac{1}{Cs+C_0}
	\]
	for some constants $C$, $C_0$, where the open interval $I$ is chosen such that $\Psi$ is well-defined and positive on $I$. We find $\widetilde{N}^2=C_0^2\Psi^2$, so we need to impose $C_0\not=0$.
	Changing coordinates as above, using $h(r)=\widetilde{N}^2$ and $r=s\Psi$, we see that the corresponding metric coefficient $h$ satisfies $h(r)=(1-Cr)^2$ indepentent of the choice of $C_0\not=0$. For $C=0$ we recover the Minkowski spacetime. Explicit computation gives
	\[
	\frac{1}{2}h''+\frac{(n-3)}{2r}h'-\frac{(n-2)}{r^2}(h-1)=\frac{n-1}{r}C,
	\]
	so Theorem \ref{thm_charakterisationphotonsurfaces} applies for $C\not=0$ ($n\ge 2$). Thus all photon surfaces with constant umbilicity factor $\lambda$ must necessarily be centered around the origin. However, in contrast to the result of Cederbaum--Galloway, this illustrates that assumption \eqref{eq_eigenvalueTRUE2} does not prevent the formation of ``off-center'' photon surfaces in general.
	If $C>0$ the corresponding spacetime of class $\SCal$ with profile $h$ satisfies the (NEC) and the (DEC) (with negative cosmological constant $\Lambda=-\frac{n(n-1)}{2}C^2$), and the spacetime possesses a degenerate Killing horizon at $r=\frac{1}{C}$. As $s\Psi:(0,\infty)\to(0,\frac{1}{C})$, only the interior of the degenerate Killing horizon is spacetime conformally flat, while the exterior is not. These spacetimes are therefore an interesting toy-model to compare the result of Cederbaum--Galloway and Theorem \ref{thm_charakterisationphotonsurfaces}, as there are a wide array of ``off-center'' photon surfaces in the interior of the degenerate horizon (which have to be conformal images of either pseudospheres or hyperplanes), while in the exterior any photon surface has to be centered in the sense of Theorem 3.8 in \cite{cedgal}. Althewhile, Theorem \ref{thm_charakterisationphotonsurfaces} implies that any photon sphere with constant umbilicity factor is necessarily centered around the origin. In fact, the centered pseudospheres $\{s=\sqrt{R^2+t^2}\}$ of radius $R$ in isotropic coordinates in the interior of the degenerate horizon retain $\lambda=\frac{1}{R}$.
	\vspace{4cm}
\begin{appendix}
	\section{Curvature identities}\label{appendix_curvature}
	\begin{lem}\label{lem_intrinsiccurvature1}
		For any warped product graph $(M_T,g^T)$, we have that
		\begin{align*}
		\Ric^T_{ss}&=-\frac{(n-1)}{2s}\frac{h_T'}{h_T},\\
		\Ric^T_{IJ}&=\left(\Ric_{g_\mathcal{N}}\right)_{IJ}-\left((n-2)h_T+\frac{1}{2}sh_T'\right)\left(g_{\mathcal{N}}\right)_{IJ},\\
		\scal^T&=s^{-2}\scal_{g_\mathcal{N}}-\frac{(n-1)}{s^2}\left((n-2)(h_T)+sh_T'\right),\\
		\frac{\Hess_Tf_T}{f_T}&=\frac{1}{2}\frac{h_T''}{h_T}\d s^2+\frac{1}{2}rh_T'g_{\mathcal{N}},\\
		\frac{\Delta_Tf_T}{f_T}&=\frac{1}{2}h_T''+\frac{(n-1)}{2s}h_T',
		\end{align*}
		where $\Ric^T$, $R^T$, $\Hess_T$, and $\Delta_T$ denote the Ricci curvature, scalar curvature, Hessian and Laplacian along $(M_T, g^T)$ respectively, $\Ric_{g_\mathcal{N}}$, and $\scal_{g_\mathcal{N}}$ denote the Ricci curvature and scalar curvature of $(\mathcal{N},g_\mathcal{N})$, respectively, and $f_T$ is defined as $f_T\definedas \sqrt{h_T}$ on $I$.
		Furthermore, for a spacetime $(\mathfrak{M},\mathfrak{g})$ of Class $\mathfrak{H}$, we have
		\begin{align*}
		\mathfrak{Rm}(\cdot,\partial_t,\cdot,\partial_t)&=f_0\Hess_0f_0(\cdot,\cdot)\\
		\mathfrak{Rm}(\cdot,\partial_r,\cdot,\partial_r)&=f_0\Hess_0f_0(\partial_r,\partial_r)\d t^2-f^{-3}\Hess_0f_{0}(\partial_I,\partial_J)\d x^I\d x^J\\
		\mathfrak{Rm}(\cdot,\partial_r,\cdot,\partial_t)&=-\frac{1}{2}h''\d t\d r\\
		\mathfrak{Rm}(\cdot,\partial_t,\cdot,\partial_r)&=-\frac{1}{2}h''\d r\d t\\
		\mathfrak{R}&=-h''-\frac{(n-1)}{r^2}\left((n-2)h+2rh'\right)+r^{-2}\scal_{g_\mathcal{N}}.
		\end{align*}
	\end{lem}
	\begin{proof}
		Assuming the warped product structure, all non-trivial Christoffel symbols are given by 
		\begin{align*}
		\leftidx{^T}\!{\Gamma}_{ss}^s&=-\frac{1}{2}\frac{h_T'}{h_T},\\
		\leftidx{^T}\!{\Gamma}_{Is}^K&=\frac{1}{s}\delta^K_I,\\
		\leftidx{^T}\!{\Gamma}_{IJ}^s&=-sh_T\left(g_{\mathcal{N}}\right)_{IJ},\\
		\leftidx{^T}\!{\Gamma}_{IJ}^K&=\leftidx{^\mathcal{N}}\!{\Gamma}_{IJ}^K,
		\end{align*}
		where $\leftidx{^\mathcal{N}}\!{\Gamma}_{IJ}^K$ denotes Christoffel symbols of $(\mathcal{N},g_{\mathcal{N}})$. By definition, the Ricci curvature components are given by the formula
		\[
		\Ric^T_{ij}=\leftidx{^T}\!{\Gamma}_{ij,k}^k-\leftidx{^T}\!{\Gamma}_{ki,j}^k+\leftidx{^T}\!{\Gamma}_{ij}^k\leftidx{^T}\!{\Gamma}_{kl}^l-\leftidx{^T}\!{\Gamma}_{ik}^l\leftidx{^T}\!{\Gamma}_{jl}^k,
		\]
		where $i,j$ denote the coordinates $\{s,x^I\}$ on $M_T$. Then, a straightforward computation yields the above identities for the Ricci curvature components. Taking the metric trace with respect to $g^T$ yields the scalar curvature
		\[	\scal^T=s^{-2}\scal_{g_\mathcal{N}}-\frac{(n-1)}{s^2}\left((n-2)(h_T)+sh_T'\right).		\]
		The Hessian $\Hess_T$ of $f_T$ on $M_T$ is given by
		\[
		\left(\Hess_Tf_{T}\right)_{ij}=f_{T,ij}-\leftidx{^T}\!{\Gamma}_{ij}^kf_{T,k}
		\]
		where $f_{T,i}\definedas \partial_if_T$. Using the above identities for the Christoffel symbols and taking the metric trace with respect to $g^T$, the identities for the Hessian and Laplacian of $f_T$ are immediate.\newline
		Computing all non-trivial Christoffel symbols on $(\mathfrak{M},\mathfrak{g})$ yields the identities for the relevant curvature components of $\mathfrak{Rm}$ in a similarly straightforward way. Having computed the Ricci curvature on $M_0$ by putting $h_T=h$, we get the expressions for $\mathfrak{Ric}$ using O'Neill's formula (Proposition 2.7 \cite{corvino}). Taking the metric trace with respect to $\mathfrak{g}$ yields the explicit formula for the scalar curvature $\mathfrak{R}$.
	\end{proof}
	We complete the proof of Lemma \ref{lem_constraintsWarpedproductgraphs} with the following two lemmas. Recall that for a tangent vector $V_T=c_1f_T\partial_s+\frac{c_2}{s}X$, we define $V_0:=c_1f\partial_r+\frac{c_s}{s}X$ tangent to $M_0$.
	\begin{lem}\label{lem_appendix1}
		For a warped product graph $(M_T,g^T,K^T)$ we have
		\begin{align*}
		\mu_T=\mu_0=\frac{1}{2}\scal_0&=\frac{1}{2}\mathfrak{R}+\frac{\Delta_0f_0}{f_0},\\
		\Ric^T(V_T,V_T)+(\tr_TK^TK^T-(K^T)^2)(V_T,V_T)&=\Ric^0(V_0,V_0).
		\end{align*}
	\end{lem}
	\begin{proof}
		For the first identity, it suffices to show $\mu_T=\frac{1}{2}R_0$. The rest is immediate or follows from Lemma \ref{lem_intrinsiccurvature1}. Since
		\[
			2\mu_T=R_T-\vert{K^T}\vert^2+(\tr_TK^T)^2
		\]
		the claims follow from Equations \eqref{eq_K_T_1}, \eqref{eq_K_T_2}, Lemma \ref{lemma_difference} and Lemma \ref{lem_intrinsiccurvature1}. The identity for $\Ric^T$ follows directly from Lemma \ref{lem_intrinsiccurvature1} and Lemma \ref{lemma_difference}.
	\end{proof}
	\begin{lem}\label{lem_riemannhessian}
		\begin{align*}
		\mathfrak{Rm}(V,\vec{n}^T,V,\vec{n}^T)=\frac{\Hess_0f(V_0,V_0)}{f}.
		\end{align*}
	\end{lem}
	\begin{proof}
		Recall that $\vec{n}=\frac{\partial_t+h\nabla_0T}{f\sqrt{1-h\btr{\nabla_0T}^2}}$. As $\nabla_0T=hT'\partial_r$, a direct computation using Lemma \ref{lem_intrinsiccurvature1} gives
		\begin{align*}
		&\mathfrak{Rm}(\cdot,\vec{n},\cdot,\vec{n})
		\,=\,&h_Tf\Hess_0f_{rr}\left(\frac{1}{h^2}\d r^2-(\d t\d r+\d r\d t)+h^2(T')^2\d t^2\right)+\frac{\Hess_0f_{IJ}}{f}\d x^I\d x^J,
		\end{align*}
		where we used Lemma \ref{lem_intrinsiccurvature1} and the explicit form \eqref{eq_graphmetric} for $h_T$. Translating $M_T$ in $t$-direction, we can extend $\partial_s$ to a vectorfield on $\mathfrak{M}$. It is straightforward to see that
		\[
		\d s=h_T\left(\frac{1}{h}\d r-hT'\d t\right), 
		\]
		so we have
		\[
		\mathfrak{Rm}(\cdot,\vec{n},\cdot,\vec{n})=\frac{f}{h_T}\Hess_0f_{rr}\d s^2+\frac{\Hess_0f_{IJ}}{f}\d x^I\d x^J.
		\]
		Thus, for $V=c_1f_T\partial_s+\frac{c_2}{s}X$
		\begin{align*}
		\mathfrak{Rm}(V,\vec{n},V,\vec{n})
		&=c_1^2f\Hess_0f_{rr}+\frac{c_2^2}{r^2}\frac{\Hess_0f(X,X)}{f}\\
		&=\frac{\Hess_0f(V_0,V_0)}{f},
		\end{align*}
		where we used the definition of $V_0$.
	\end{proof}
\end{appendix}
\bibliographystyle{plain}	
\bibliography{bib_stcmchyperboloids}

\nopagebreak
\end{document}